\documentclass[letter, 12pt]{article}  

\usepackage{amsthm}

\usepackage{multirow}

\usepackage[cmex10]{amsmath}
\usepackage{amssymb} 

\usepackage{cases} 

\usepackage{fixltx2e}

\usepackage[subrefformat=parens]{subcaption}
\usepackage{graphicx} 
\usepackage{epstopdf}

\usepackage{xcolor}

\usepackage{cite}

\usepackage{algorithm}
\usepackage{algorithmic}

\newtheorem{theorem}{Theorem}%
\newtheorem{proposition}{Proposition}%
\newtheorem{corollary}{Corollary}%
\newtheorem{assumption}{Assumption}%
\newtheorem{remark}{Remark}

\definecolor{myGreen}{rgb}{ 0, 0.7, 0.3 }
\definecolor{myBlue}{rgb}{ 0, 0.4, 1 }
\definecolor{myPurple}{rgb}{ 0.7, 0, 0.3 }
\definecolor{myGray}{gray}{ 0.7 }
\newcommand{\cR}{\textcolor{red}}

\usepackage{pgfplots}
\pgfplotsset{compat=newest}
\usepackage{tikz}
\usetikzlibrary{positioning}
\usetikzlibrary{arrows}

\newcommand{\MyHighlight}[1]{\textbf{#1}}

\title{
Theoretical Analysis of Heteroscedastic Gaussian Processes with Posterior Distributions}

\author{Yuji Ito%
\thanks{%
This work has been submitted to a journal or conference for possible publication. Copyright may be transferred without notice, after which this version may no longer be accessible.
This work was not supported by any organization. 
{We would like to thank Editage (www.editage.jp) for English language editing.}
}
\thanks{Yuji Ito is the corresponding author and with Toyota Central R\&D Labs., Inc., 41-1 Yokomichi, Nagakute-shi, Aichi 480-1192, Japan {\tt\small ito-yuji@mosk.tytlabs.co.jp}}%
}

\date{\empty}

\begin{document}
\maketitle
\thispagestyle{empty}
\pagestyle{empty}

\begin{abstract}
This study introduces a novel theoretical framework for analyzing heteroscedastic Gaussian processes (HGPs) that identify unknown systems in a data-driven manner.
Although HGPs effectively address the heteroscedasticity of noise in complex training datasets,
calculating the exact posterior distributions of the HGPs is challenging, as these distributions are no longer multivariate normal.
This study derives the exact means, variances, and cumulative distributions of the posterior distributions.
Furthermore, the derived theoretical findings are applied to a chance-constrained tracking controller.
After an HGP identifies an unknown disturbance in a plant system, the controller can handle chance constraints regarding the system despite the presence of the disturbance.

\end{abstract}

\newcommand{\SymColor}[1]{\cR{#1}}


\newcommand{\PDFratio}{\SymColor{\mathcal{R}({\slnfv{1:\NumData}} , {\lnfv_{1:\NumData}})}}

\newcommand{\notationCollection}{\SymColor{(i)}}
\newcommand{\notationDiag}{\SymColor{(ii)}}
\newcommand{\notationEl}{\SymColor{(iii)}}
\newcommand{\notationExpect}{\SymColor{(iv)}}
\newcommand{\notationCondExpect}{\SymColor{(v)}}
\newcommand{\notationNormalDist}{\SymColor{(vi)}}


\newcommand{\JitterVal}{\SymColor{\varepsilon}}

\newcommand{\pdTransMat}{\SymColor{R}}

\newcommand{\sysx}[2]{  \El{\SymColor{\xi}(#1)}{#2}    }
\newcommand{\sysxV}[1]{\SymColor{\SymColor{\xi}(#1)}}
\newcommand{\sysu}[1]{\SymColor{u(#1)}}

\newcommand{\sysrMargin}{\SymColor{\overline{r}}}
\newcommand{\sysr}[2]{  \El{\SymColor{r}(#1)}{#2}    }
\newcommand{\sysrV}[1]{\SymColor{r}(#1)}
\newcommand{\sysv}[1]{\SymColor{v(#1)}}
\newcommand{\sysPreu}[1]{\SymColor{\hat{u}(#1)}}

\newcommand{\MyT}{\SymColor{t}}
\newcommand{\MydT}{\SymColor{\tau}}

\newcommand{\Thorizon}{  \SymColor{T}   }

\newcommand{\FBgain}{  \SymColor{\kappa}   }

\newcommand{\hypMkerfm}{\SymColor{\nu}^{\fm}}
\newcommand{\hypSkerfm}[1]{\SymColor{\varrho}^{\fm}_{#1}}


\newcommand{\unPDF}{\SymColor{\widetilde{p}}}
\newcommand{\unConstPDF}{\SymColor{Z}}
\newcommand{\PDF}{\SymColor{p}}
\newcommand{\proposalPDF}{\SymColor{q}}

\newcommand{\CoveringNum}{\SymColor{\mathcal{C}(\LipschitzGrid,\CompactDomX)}}

\newcommand{\UniErrDelta}{\SymColor{\rho}}
\newcommand{\UniErrCoefBeta}{\SymColor{\beta(\LipschitzGrid,\UniErrDelta)}}

\newcommand{\LipschitzInputClosest}{\SymColor{\xi_{\Input}}}
\newcommand{\Lipschitzfm}{\SymColor{\mathcal{L}_{\fm}}}
\newcommand{\Lipschitzgpmfm}{\SymColor{\mathcal{L}_{\gpmfm}}}
\newcommand{\Lipschitzgpsdfm}{\SymColor{\mathcal{L}_{\gpsdfm}}}

\newcommand{\Lipschitzkernel}{\SymColor{\mathcal{L}_{k}}}
\newcommand{\LipschitzGrid}{\SymColor{\nu}}
\newcommand{\LipschitzCompactDomX}{\SymColor{\widetilde{\mathcal{X}}_{\LipschitzGrid}}}
\newcommand{\CompactDomX}{\SymColor{\widetilde{\mathcal{X}}}}

\newcommand{\sparseControllerL}{\SymColor{u_{\mathrm{l}}}}
\newcommand{\sparseControllerU}{\SymColor{u_{\mathrm{u}}}}

\newcommand{\CDFdelta}{\SymColor{\delta}  }
\newcommand{\CDFlev}{\SymColor{\gamma}  }

\newcommand{\CDFlevL}{\SymColor{\CDFlev_\mathrm{l}}  }
\newcommand{\CDFlevU}{\SymColor{\CDFlev_\mathrm{u}}  }

\newcommand{\CDFlevSpecialU}{\SymColor{\CDFlev_\mathrm{u\ast}}  }
\newcommand{\CDFlevSpecialL}{\SymColor{\CDFlev_\mathrm{l\ast}}  }
\newcommand{\CDFdeltaSpecial}{\SymColor{\delta_{\ast}}  }

\newcommand{\CDFlevSpecialTerm}{\SymColor{\eta}}

\newcommand{\CDFlevStandardNormal}{\SymColor{\widetilde{\gamma}}  }

\newcommand{\Qfunc}{\SymColor{Q}  }

\newcommand{\erfc}{\SymColor{\mathrm{erfc}}  }

\newcommand{\fmForInt}{\SymColor{\widetilde{\fm}}}

\newcommand{\gpmfm}{\SymColor{\mu}_{\fm}}
\newcommand{\gpsdfm}{\SymColor{\sigma}_{\fm}}

\newcommand{\gpmGhfm}{\SymColor{\mu}_{\fm|\lnfv}}
\newcommand{\gpsdGhfm}{\SymColor{\sigma}_{\fm|\lnfv}}

\newcommand{\gpsdGramCollection}{\SymColor{\mathcal{K}}}

\newcommand{\wildcard}{\SymColor{\bullet}}

\newcommand{\IDEl}{\SymColor{i}}
\newcommand{\IDbEl}{\SymColor{j}}
\newcommand{\IDcEl}{\SymColor{i^{\prime}}}
\newcommand{\IDdEl}{\SymColor{j^{\prime}}}

\newcommand{\Identity}{\SymColor{{I}}}

\newcommand{\za}{\SymColor{a}}
\newcommand{\zb}{\SymColor{b}}
\newcommand{\zc}{\SymColor{c}}
\newcommand{\zd}{\SymColor{d}}
\newcommand{\ze}{\SymColor{e}}
\newcommand{\zf}{\SymColor{f}}
\newcommand{\zg}{\SymColor{g}}
\newcommand{\sparseonehotv}{\SymColor{\boldsymbol{h}}}
\newcommand{\zi}{\SymColor{i}}
\newcommand{\zj}{\SymColor{j}}
\newcommand{\zk}{\SymColor{k}}
\newcommand{\zl}{\SymColor{l}}
\newcommand{\zm}{\SymColor{m}}
\newcommand{\zn}{\SymColor{n}}
\newcommand{\zo}{\SymColor{o}}
\newcommand{\zp}{\SymColor{p}}
\newcommand{\zq}{\SymColor{q}}
\newcommand{\zr}{\SymColor{r}}
\newcommand{\zs}{\SymColor{s}}
\newcommand{\zbs}{\SymColor{s^{\prime}}}
\newcommand{\zt}{\SymColor{t}}
\newcommand{\zzzu}{\SymColor{\boldsymbol{u}}}
\newcommand{\zw}{\SymColor{w}}
\newcommand{\zx}{\SymColor{x}}
\newcommand{\zy}{\SymColor{y}}
\newcommand{\zz}{\SymColor{z}}

\newcommand{\zzzB}{\SymColor{\boldsymbol{B}}}
\newcommand{\zC}{\SymColor{C}}
\newcommand{\zD}{\SymColor{D}}
\newcommand{\zF}{\SymColor{F}}
\newcommand{\zG}{\SymColor{G}}
\newcommand{\zH}{\SymColor{H}}
\newcommand{\zJ}{\SymColor{J}}
\newcommand{\zzzK}{\SymColor{\boldsymbol{K}}}
\newcommand{\zzzY}{\SymColor{\boldsymbol{Y}}}
\newcommand{\zM}{\SymColor{M}}
\newcommand{\zN}{\SymColor{N}}
\newcommand{\zO}{\SymColor{O}}
\newcommand{\zP}{\SymColor{P}}

\newcommand{\zS}{\SymColor{S}}
\newcommand{\zT}{\SymColor{T}}
\newcommand{\zU}{\SymColor{U}}

\newcommand{\zalpha}{\SymColor{\alpha}}
\newcommand{\zbeta}{\SymColor{\beta}}
\newcommand{\zgamma}{\SymColor{\gamma}}

\newcommand{\zrho}{\SymColor{\rho}}

\newcommand*{\VEC}[2][]{\SymColor{\mathrm{vec }#1(} #2 \SymColor{#1)}}
\newcommand{\DIAG}[1]{\SymColor{\mathrm{diag}}(#1)}
\newcommand{\eDIAG}[1]{\SymColor{\mathrm{ediag}}(#1)}

\newcommand{\Zdiag}{\SymColor{\mathrm{diag}}}
\newcommand{\Zeig}[1]{\SymColor{\lambda_{#1}}}
\newcommand{\Zsingularmax}{\SymColor{\sigma_{\max}}}
\newcommand{\Zeigmin}{\SymColor{\lambda_{\min}}}
\newcommand{\Zeigmax}{\SymColor{\lambda_{\max}}}

\newcommand{\ZsetY}{\SymColor{\mathcal{Y}}}

\newcommand{\Zfro}{\SymColor{\mathrm{F}}}







\newcommand*{\NotationVec}{\SymColor{{\nu}}}
\newcommand*{\NotationMat}{\SymColor{{\Sigma}}}
\newcommand*{\NotationSymMat}{\SymColor{{\Gamma}}}

\newcommand*{\NotationSto}{\SymColor{{\theta}}}
\newcommand*{\NotationFunc}{\SymColor{{\varphi}}}
\newcommand*{\IDNotation}{\SymColor{i}}
\newcommand*{\IDbNotation}{\SymColor{j}}
\newcommand*{\DimANotation}{\SymColor{a}}
\newcommand*{\DimBNotation}{\SymColor{b}}
\newcommand*{\NotationA}{\SymColor{1}}
\newcommand*{\NotationB}{\SymColor{2}}


%


\newcommand{\opQE}{\SymColor{\mathbb{E}_{\proposalPDF}}}
\newcommand{\opE}{\SymColor{\mathbb{E}}}
\newcommand{\opV}{\SymColor{\mathbb{V}}}

\newcommand*{\El}[2]{\SymColor{[}{#1}\SymColor{]}_{#2}}

\newcommand{\iEig}[2]{\SymColor{\lambda}_{#1}(#2)}

\newcommand*{\MyTop}{\SymColor{\top}}

\newcommand*{\TRACE}[1]{\SymColor{\mathrm{tr}}(#1)}


\newcommand{\DomX}{\SymColor{\mathcal{X}}}
\newcommand{\DimX}{\SymColor{n}}

\newcommand{\Input}{\SymColor{x}}
\newcommand{\Noise}{\SymColor{w}}
\newcommand{\Output}{\SymColor{y}}

\newcommand{\Dataset}[1]{\SymColor{\mathcal{D}}_{#1}}

\newcommand{\NumSamp}{\SymColor{S}}
\newcommand{\IDsamp}{\SymColor{s}}
\newcommand{\IDbsamp}{\SymColor{s^{\prime}}}

\newcommand{\NumMC}{\SymColor{M}}
\newcommand{\IDMC}{\SymColor{m}}

\newcommand{\AnyNumData}{\SymColor{\underline{D}}}
\newcommand{\NumData}{\SymColor{D}}
\newcommand{\IDt}{\SymColor{d}}
\newcommand{\IDbt}{\SymColor{d^{\prime}}}

\newcommand{\fm}{\SymColor{f}}
\newcommand{\fv}{\SymColor{g}}
\newcommand{\lnfv}{\SymColor{h}}

\newcommand{\OutSamp}[2]{\Output_{#1,#2}}
\newcommand{\OutSampMean}[1]{\SymColor{\Output_{#1}^{\mathrm{E}}}}
\newcommand{\OutSampVar}[1]{ \SymColor{\Output_{#1}^{\mathrm{V}}}}
\newcommand{\OutSampVarForChangeOfVar}{ \SymColor{\widetilde{y}}}

\newcommand{\LCSdist}{\SymColor{ p_{ \ln \chi^{2} }  }}

\newcommand{\gpmlnfv}{\SymColor{\mu}_{\lnfv} }
\newcommand{\gpcovlnfv}{\SymColor{\Sigma}_{\lnfv} }

\newcommand{\slnfv}[1]{\SymColor{z_{#1}}}
\newcommand{\lnfvvarSQRT}{\SymColor{\omega_{\lnfv}}}
\newcommand{\lnfvvar}{\lnfvvarSQRT^{2}}

\newcommand{\kerfm}{\SymColor{k}}
\newcommand{\kerfv}{\SymColor{\ell}}

\newcommand{\kGramfm}{\SymColor{K}}
\newcommand{\kGramfv}{\SymColor{L}}
\newcommand{\kvecfm}{\SymColor{k}}
\newcommand{\kvecfv}{\SymColor{\ell}}

\newcommand*{\NkGramfm}{\SymColor{\widetilde{K}}}

\newcommand{\chiPDF}[1]{\SymColor{\chi_{#1}^{2}}}
\newcommand{\chiSamp}[1]{\SymColor{\zeta}_{#1}}
\newcommand*{\polygamma}[1]{\SymColor{\psi^{(#1)}}}





\renewcommand{\SymColor}[1]{{#1}}

\section{Introduction}\label{sec_intro}

Gaussian process regression (GPR) \cite{Rasmussen06} is a successful approach to identify complex systems as Gaussian processes (GPs) with quantifying their uncertainties.
The GPR has attracted attention of the control community, leading to the development of data-driven approaches such as stability analyses \cite{Berkenkamp16CDC,ItoCDC19}, nonlinear optimal control \cite{Boedecker14,PILCO15PAMIRN268,ItoIEICE22,ItoAccess20}, H-infinity control \cite{Fujimoto18JCMSI,ItoAccess20}, reinforcement learning \cite{Berkenkamp17}, model predictive control (MPC) \cite{Guzman21RAL,Kocijan03}, and controller tuning with Bayesian optimization \cite{Brunzema22CDC}.
Control applications using GPs including robotics \cite{NguyenTuong08RN269,Berkenkamp16ICRARN270} and aircraft \cite{Hemakumara13RN271} have also been developed.

Heteroscedastic GPs (HGPs) can treat the heteroscedasticity of noise in training datasets, unlike standard GPs that assume homoscedastic noise.
This heteroscedasticity arises from the complexity of regression tasks.
For instance, stochastic kriging has encountered heteroscedastic simulation responses \cite{Chen17EJOR,Xie20WSC}.
Feynman-Kac samples have been considered as heteroscedastic data points in solutions to partial differential equations \cite{Inoue24arXiv}.
Performance prediction of nonlinear MPC using GPs has focused on heteroscedastic observations \cite{Spenger23IFACWC}.
Benefits of heteroscedasticity of GPs have been discussed  in stochastic optimization for trajectory planning \cite{Petrovic19ECMR}.

However, analytically realizing heteroscedastic GPR (HGPR) is infeasible due to the difficulty in obtaining the exact posterior distributions of HGPs while those of homoscedastic GPs have been given by simple normal distributions \cite{Rasmussen06}.
The posterior distributions correspond to the predictions of target functions under GP settings.
This challenge arises from the uncertainty of heteroscedastic noise, which is difficult to estimate from datasets in a Bayesian manner.
Many approaches have approximated the noise variance using deterministic models without uncertainty, conflicting with Bayesian estimation.
Such deterministic approximations have used
polynomial chaos expansions \cite{Polke24},
the posterior mean of a GP \cite{Menner23CCTA,Caravagna16CMSB},
most likely noise levels \cite{Zhang20TSP,Kersting07ICML},
nonlinear basis functions such as polynomials \cite{Guzman21RAL,Boukouvalas14CSDA}, 
a (weighted unbiased) sample variance at each data point and/or its neighbors \cite{Inoue24arXiv,Feinstein24arXiv,Spenger23IFACWC,Caravagna16CMSB}, 
and
worst-case variance values \cite{Makarova21NeurIPS}.
Other methods have approximated the posterior distributions of HGPs \cite{Munoz14TNNLS,Munoz16PAMI}.
Although sampling-based methods have numerically computed the posterior distributions \cite{Heinonen16,Goldberg97NIPS}, their analytical expressions have not been provided.

To address the aforementioned challenge, this study presents novel theoretical results for HGPR, which have not been investigated in the existing work.
We derive the mean, variance, and cumulative distribution of the posterior distribution of an HGP in exact forms, utilizing expectations regarding homoscedastic noise.
A method based on autonormalized importance sampling is proposed to calculate these expectations.
Additionally, this study proposes a chance-constrained tracking controller based on the derived results.
Once HGPR identifies an unknown disturbance within a plant system, the proposed controller can manage a chance constraint regarding the system even if the disturbance is involved.
A numerical simulation is conducted to validate the effectiveness of the proposed controller.

\textbf{\textit{Notation:}} This study uses the following notation.
\begin{itemize}

	\item[\notationCollection]
	${\NotationVec}_{1:\NumData}:=[{\NotationVec}_{1}, {\NotationVec}_{2}, \dots, {\NotationVec}_{\NumData}]^{\MyTop}  \in \mathbb{R}^{\NumData\times \DimANotation}$: 
	the collection of variables ${\NotationVec}_{\IDt} \in \mathbb{R}^{\DimANotation}$ for all $\IDt \in \{1,2,\dots, \NumData\}$

	\item[\notationDiag]
	$\Zdiag(\NotationVec)$: the $\DimANotation \times \DimANotation$ diagonal matrix of which diagonal components are $\NotationVec\in \mathbb{R}^{\DimANotation}$
	
	\item[\notationEl]
	$\El{\NotationVec}{\IDNotation}$: the $\IDNotation$-th component of a vector $\NotationVec \in \mathbb{R}^{\DimANotation}$

	\item[\notationExpect]
	$\opE[\NotationFunc(\NotationSto)]$ (resp. $\opV[\NotationFunc(\NotationSto)]$): 
	the expectation (resp. variance) of a function $\NotationFunc(\NotationSto)$ regarding a random variable $\NotationSto$
	
	\item[\notationCondExpect]
	$\opE[\NotationFunc(\NotationSto_{1}, \NotationSto_{2})|\NotationSto_{2}]$ (resp. 
	$\opV[\NotationFunc(\NotationSto_{1}, \NotationSto_{2})|\NotationSto_{2}]$): 
	the conditional expectation (resp. conditional variance) of a function $\NotationFunc(\NotationSto)$ regarding a random variable $\NotationSto_{1}$ given $\NotationSto_{2}$

	\item[\notationNormalDist]
	$\mathcal{N}(\NotationSto | \NotationVec, \NotationMat )$:
	the multivariate normal distribution with the mean $\NotationVec$ and covariance $\NotationMat$ that a random variable $\NotationSto$ obeys,
	where this is often denoted by $\mathcal{N}( \NotationVec, \NotationMat )$.

\end{itemize}

\section{Problem setting}\label{sec_problem}

Consider two unknown functions $\fm: \DomX \to \mathbb{R}$ and $\fv: \DomX \to (0,\infty)$ for a given set $\DomX \subseteq \mathbb{R}^{\DimX}$.
We assume that a training dataset is given to identify $\fm$ and $\fv$ as follows:
Let $\Input_{\IDt} \in \DomX$ for $\IDt \in \{ 1,2,\dots,\NumData\}$ be $\NumData$ training inputs that satisfy $\Input_{\IDt}\neq \Input_{\IDbt}$ for every $\IDt\neq\IDbt$.
For each $\Input_{\IDt}$, we can observe $\NumSamp$ outputs $\Output_{\IDt,\IDsamp}$ for $\IDsamp \in \{ 1,2,\dots,\NumSamp\}$ with $\NumSamp \geq 1$ that obey
\begin{align}
	\Output_{\IDt,\IDsamp}
	&=  \fm(\Input_{\IDt}) + \Noise_{\IDt,\IDsamp}
	,\label{eq:def_output}
	\\
	\Noise_{\IDt,\IDsamp} 
	\big| \fv(\Input_{\IDt}) 
	& \sim \mathcal{N}(0, \fv(\Input_{\IDt})^{2} )
	,\label{eq:def_noise}
\end{align}
where $\Noise_{\IDt,\IDsamp}\in \mathbb{R}$ is a random noise term that is independently and identically distributed (i.i.d.) concerning $(\IDt,\IDsamp)$.
This implies  that $\Output_{\IDt,\IDsamp}$ obeys $\mathcal{N}(\fm(\Input_{\IDt}), \fv(\Input_{\IDt})^{2} )$.

This study utilizes HGPR to estimate the unknown functions $\fm$ and $\fv$ concurrently.
Because the HGPR focuses on functions whose ranges can contain both negative and positive values, the following logarithm $\lnfv: \DomX \to \mathbb{R}$ is identified instead of strictly positive $\fv$:
\begin{align}
	\lnfv(\Input):=\ln \fv(\Input)^{2}
	.\label{eq:def_lnfv}
\end{align}
We state the following main problem.

\textit{\textbf{Main problem:}} 
Estimate the unknown function $\fm: \DomX \to \mathbb{R}$ for the given dataset, namely, find  the mean, variance, and cumulative distribution function of the posterior of $\fm$.

\section{Proposed method with HGPR}\label{sec_method}

We address the main problem by utilizing HGPR, which identifies the unknown functions $\fm$ and $\fv$.
Section \ref{sec_HGPR_theory} presents the mean, variance, and cumulative distribution of the posterior of $\fm$, which correspond solutions to the main problem theoretically.
Section \ref{sec_HGPR_implementation} outlines a practical approach to computing the derived posterior information. 

\subsection{Theoretical analyses of HGPR}\label{sec_HGPR_theory}

The HGPR framework estimates $\fm$ and $\lnfv$ under the following assumption.

\begin{assumption}[{\MyHighlight{Two Gaussian processes}}]\label{ass_GPs}
	The functions $\fm$ and $\lnfv$ obey GPs with strictly positive definite kernels $\kerfm : \DomX\times \DomX \to \mathbb{R} $ and $\kerfv : \DomX\times \DomX \to \mathbb{R} $, respectively.
	Specifically, for any natural number $\AnyNumData$ and any distinct $\Input_{1:\AnyNumData}
	=[\Input_{1}, \dots, \Input_{\AnyNumData}]^{\MyTop}$ satisfying $\Input_{\IDt}\neq \Input_{\IDbt} $ for $\IDt\neq \IDbt$, we have
	$\kGramfm_{\AnyNumData} \succ 0$, $\kGramfv_{\AnyNumData} \succ 0$, and 
	\begin{align}
		\forall \lnfv,
		\quad
		[\fm(\Input_{1}), \dots, \fm(\Input_{\AnyNumData})]^{\MyTop}
		|
		\Input_{1:\AnyNumData}
		&\sim
		\mathcal{N} ( 0 ,\kGramfm_{\AnyNumData} )
		,\nonumber
		\\
		\forall \fm,
		\quad
		[\lnfv(\Input_{1}), \dots, \lnfv(\Input_{\AnyNumData})]^{\MyTop}
		|
		\Input_{1:\AnyNumData}
		&\sim
		\mathcal{N} ( 0 ,\kGramfv_{\AnyNumData} )
		,\nonumber
	\end{align}
	where the following gram matrices $(\kGramfm_{\AnyNumData},\kGramfv_{\AnyNumData})$ and vectorized kernels $( \kvecfm_{\AnyNumData}(\Input), \kvecfv_{\AnyNumData}(\Input))$ are utilized:
	\begin{align}		
		\kGramfm_{\AnyNumData}
		&:=
		[\kvecfm_{\AnyNumData}(\Input_{1}),\kvecfm_{\AnyNumData}(\Input_{2}),\dots,\kvecfm_{\AnyNumData}(\Input_{\AnyNumData})]\in \mathbb{R}^{\AnyNumData \times \AnyNumData}
		,\nonumber\\
		\kGramfv_{\AnyNumData}
		&:=
		[\kvecfv_{\AnyNumData}(\Input_{1}),\kvecfv_{\AnyNumData}(\Input_{2}),\dots,\kvecfv_{\AnyNumData}(\Input_{\AnyNumData})]\in \mathbb{R}^{\AnyNumData \times \AnyNumData}
		,\nonumber\\
		\kvecfm_{\AnyNumData}(\Input)
		&:=[
		\kerfm(\Input,\Input_{1}),\kerfm(\Input,\Input_{2}),\dots,\kerfm(\Input,\Input_{\AnyNumData})
		]^{\MyTop} \in \mathbb{R}^{\AnyNumData }
		,\nonumber\\
		\kvecfv_{\AnyNumData}(\Input)
		&:=[
		\kerfv(\Input,\Input_{1}),\kerfv(\Input,\Input_{2}),\dots,\kerfv(\Input,\Input_{\AnyNumData})
		]^{\MyTop}\in \mathbb{R}^{\AnyNumData }
		.\nonumber
	\end{align}
\end{assumption}

A well-known instance of a strictly positive definite kernel is a squared exponential kernel \cite[Section 3.1]{Posa24CSTM}, as introduced in Section \ref{sec_sim_setting}.
We derive the following theorem to obtain the posterior mean and variance of $\fm$, with the dataset ${\Dataset{}}$ and other symbols defined using Notations {\notationCollection} and {\notationDiag} as follows:
\begin{align}
	{\Dataset{}}
	&:=[ \Input_{1:\NumData} , {\OutSampMean{1:\NumData}} ]
	,\nonumber\\
	{\OutSampMean{\IDt}} &:=\frac{1}{\NumSamp}\sum_{\IDsamp=1}^{\NumSamp} {\OutSamp{\IDt}{\IDsamp}}
	,\label{eq:OutSampMean}
	\\
	\lnfv_{\IDt}&:=\lnfv(\Input_{\IDt})
	,\nonumber\\
	\NkGramfm(\lnfv_{1:\NumData})	
	&:=\kGramfm_{\NumData} + \frac{1}{\NumSamp} \Zdiag(
	[\exp \lnfv_{1}, \dots, \exp \lnfv_{\NumData}]^{\MyTop} 	
	)
	.\nonumber
\end{align}

\begin{theorem}[{\MyHighlight{The posterior mean and variance of $\fm$}}]\label{thm:GP_fm}
	Suppose Assumption \ref{ass_GPs}.
	For any $\Input \in \DomX$, the posterior mean $\opE[\fm(\Input) |  \Input, {\Dataset{}}  ]$ and variance $\opV[\fm(\Input) |  \Input, {\Dataset{}}  ]$ of $\fm$ given $(\Input, {\Dataset{}})$ are expressed as follows:
	\begin{align}
		\opE[\fm(\Input) |  \Input, {\Dataset{}}  ]
		&=
		\kvecfm_{\NumData}(\Input)^{\MyTop} \opE[ \NkGramfm(\lnfv_{1:\NumData})^{-1}  |   {\Dataset{}}] {\OutSampMean{1:\NumData}}
		,\nonumber
		\\
		\opV[\fm(\Input) |  \Input, {\Dataset{}}  ]
		&=
		\kerfm(\Input,\Input)
		+
		\kvecfm_{\NumData}(\Input)^{\MyTop}
		\gpsdGramCollection( {\Dataset{}})
		\kvecfm_{\NumData}(\Input)
		.\nonumber
	\end{align}
	where
	\begin{align}
		\gpsdGramCollection( {\Dataset{}})
		&:=
		\opE[
		\NkGramfm(\lnfv_{1:\NumData})^{-1}   {\OutSampMean{1:\NumData}} {\OutSampMean{1:\NumData}}^{\top} \NkGramfm(\lnfv_{1:\NumData})^{-1} 
		|   {\Dataset{}}   ]
		\nonumber\\&\quad
		-
		\opE[ \NkGramfm(\lnfv_{1:\NumData})^{-1}  |   {\Dataset{}}] 
		{\OutSampMean{1:\NumData}} {\OutSampMean{1:\NumData}}^{\top}
		\opE[ \NkGramfm(\lnfv_{1:\NumData})^{-1}  |   {\Dataset{}}] 	
		\nonumber\\&\quad
		-
		\opE[ 	\NkGramfm(\lnfv_{1:\NumData})^{-1}  	|   {\Dataset{}}   ]
		.\label{eq:def_gpsdGramCollection}
	\end{align}
\end{theorem}
\begin{proof}
	The proof is described in Appendix \ref{pf:GP_fm}. 
\end{proof}

\begin{remark}[{\MyHighlight{Contributions of Theorem \ref{thm:GP_fm}}}]
Theorem \ref{thm:GP_fm} elucidates the explicit forms of the posterior mean $\opE[\fm(\Input) |  \Input, {\Dataset{}}  ]$ and variance $\opV[\fm(\Input) |  \Input, {\Dataset{}}  ]$.
This posterior information facilitates the prediction of the unknown function $\fm(\Input)$ even in the presence of heteroscedastic noise.
The computation of expectations $	\opE[ 	\dots 	|   {\Dataset{}}   ]$ is detailed in Section \ref{sec_HGPR_implementation}.
\end{remark}

A crucial issue in HGPR is that the posterior distribution of $\fm$ given $(\Input, {\Dataset{}})$ is generally not a multivariate normal distribution, unlike the posteriors of homoscedastic GPs that are simple multivariate normal distributions.
We derive the cumulative distribution function of the non-normal posterior distribution, using the  complementary error function $\erfc: \mathbb{R} \to [0,2]$:
\begin{align}
\erfc(\CDFlevStandardNormal)
&:=
\frac{2}{\sqrt{\pi}} 
\int_{\CDFlevStandardNormal}^{\infty} \exp(-{\fmForInt}^{2}) \mathrm{d} \fmForInt
.\nonumber
\end{align}

\begin{theorem}[{\MyHighlight{The cumulative distribution function of $\fm$}}]\label{thm:GP_prob_fm}
	Suppose Assumption \ref{ass_GPs}.
	For  any $\Input \in \DomX$ and any $\CDFlev \in \mathbb{R}$, we have
	\begin{align}
		\mathrm{Pr}(      \fm(\Input)  \leq \CDFlev   |  \Input, {\Dataset{}})
		= 1-\CDFdelta(\CDFlev,\Input, {\Dataset{}}) 
		,\nonumber
	\end{align}
	where
	\begin{align}		
		&\CDFdelta(\CDFlev,\Input, {\Dataset{}}) :=
		\nonumber\\&
		\frac{1}{2}
		\opE\Big[
		\erfc\Big(
		\frac{ \CDFlev - \kvecfm_{\NumData}(\Input)^{\MyTop}  \NkGramfm(\lnfv_{1:\NumData})^{-1}   {\OutSampMean{1:\NumData}}  
		}{\sqrt{2} ( \kerfm(\Input,\Input)
			-  \kvecfm_{\NumData}(\Input)^{\MyTop}
			\NkGramfm(\lnfv_{1:\NumData})^{-1}
			\kvecfm_{\NumData}(\Input)   ) }
		\Big)
		\Big|
		\Input, {\Dataset{}}
		\Big]
		.\nonumber
	\end{align}
\end{theorem}
\begin{proof}
The proof is described in Appendix \ref{pf:GP_prob_fm}.
\end{proof}
\begin{remark}[{\MyHighlight{Contributions of Theorem \ref{thm:GP_prob_fm}}}]
Theorem \ref{thm:GP_prob_fm} presents the exact cumulative distribution of the posterior of $\fm$ given $(\Input, {\Dataset{}})$, using the conditional expectation of $\erfc$.
We apply the derived cumulative distribution to evaluating chance constraints for control applications in Section \ref{sec_sim}.
The complementary error function $\erfc$ can be computed with arbitrary precision \cite[Lemma 1]{Chevillard12IC}.
The computation of expectations $	\opE[ 	\dots 	|   {\Dataset{}}   ]$ is detailed in Section \ref{sec_HGPR_implementation}.
\end{remark}

\begin{corollary}[{\MyHighlight{The probability of $\fm$ being in an interval}}]\label{thm:GP_prob_fm_special}
	Suppose Assumption \ref{ass_GPs}.
	For any $\Input\in\DomX$, any $\CDFlevL \in \mathbb{R}$, and any $\CDFlevU \in \mathbb{R} $ such that $\CDFlevL < \CDFlevU$, we have
	\begin{align}
		\mathrm{Pr}(   \CDFlevL <    \fm(\Input)   \leq \CDFlevU   |  \Input, {\Dataset{}})
		= \CDFdelta(\CDFlevL,\Input, {\Dataset{}})-\CDFdelta(\CDFlevU,\Input, {\Dataset{}}) 
		.\nonumber
	\end{align}
\end{corollary}
\begin{proof}
	Theorem \ref{thm:GP_prob_fm} implies 
	$
		\mathrm{Pr}(   \CDFlevL  <   \fm(\Input)   \leq \CDFlevU   |  \Input, {\Dataset{}})
		= 
		\mathrm{Pr}(      \fm(\Input)  \leq \CDFlevU   |  \Input, {\Dataset{}})
		-
		\mathrm{Pr}(      \fm(\Input)  \leq \CDFlevL   |  \Input, {\Dataset{}})
		=
		\CDFdelta(\CDFlevL,\Input, {\Dataset{}})-\CDFdelta(\CDFlevU,\Input, {\Dataset{}}) 
	$.	
	This completes the proof.	
\end{proof}

\subsection{Practical implementation for HGPR}\label{sec_HGPR_implementation}

Theorems \ref{thm:GP_fm} and \ref{thm:GP_prob_fm} necessitate the computation of the expectation $\opE[ \dots |   {\Dataset{}}]= \int (\dots)   \PDF( \lnfv_{1:\NumData}   | {\Dataset{}} ) \mathrm{d} \lnfv_{1:\NumData}$ concerning $\lnfv_{1:\NumData}$ given the dataset $ {\Dataset{}}$.
Because $\PDF( \lnfv_{1:\NumData}   | {\Dataset{}} )$ is not obtained in an analytical form, we propose the use of autonormalized importance sampling \cite{Agapiou2017SS}.

\begin{remark}[{\MyHighlight{Autonormalized importance sampling}}]\label{thm:SNIS}
For any $\unConstPDF>0$,
consider an unnormalized PDF $\unPDF( \lnfv_{1:\NumData}   | {\Dataset{}} ):= \unConstPDF \PDF( \lnfv_{1:\NumData}   | {\Dataset{}} )$ and a continuous proposal PDF $\proposalPDF: \mathbb{R}^{\NumData} \to (0,\infty)$.
For any continuous function $\NotationFunc: \mathbb{R}^{\NumData} \to \mathbb{R}$ and $\NumMC\geq 1$, 
its expectation can be approximated as follows \cite{Agapiou2017SS}:
\begin{align}
	\opE[ \NotationFunc(\lnfv_{1:\NumData}) |   {\Dataset{}}]
	&
	=
	\opQE \Big[
	\NotationFunc(\lnfv_{1:\NumData}) 
	\frac{ \unPDF( \lnfv_{1:\NumData}   | {\Dataset{}} ) }{  \proposalPDF( \lnfv_{1:\NumData}  )  }
	 \Big]
	\Big/
	\opQE \Big[
	\frac{ \unPDF( \lnfv_{1:\NumData}   | {\Dataset{}} ) }{  \proposalPDF( \lnfv_{1:\NumData}  )  }
	\Big]
	\nonumber\\&
	\approx
	\sum_{\IDMC=1}^{\NumMC}
	\NotationFunc(\lnfv_{1:\NumData}^{(\IDMC)}) 
	\frac{ \unPDF( \lnfv_{1:\NumData}^{(\IDMC)}   | {\Dataset{}} ) }{  \proposalPDF( \lnfv_{1:\NumData}^{(\IDMC)}  )  }
	\Big/
	\sum_{\IDMC=1}^{\NumMC}
	\frac{ \unPDF( \lnfv_{1:\NumData}^{(\IDMC)}   | {\Dataset{}} ) }{  \proposalPDF( \lnfv_{1:\NumData}^{(\IDMC)}  )  }
	,\label{eq:concept_SNIS}
\end{align}
where $\opQE[ \dots]= \int (\dots ) \proposalPDF( \lnfv_{1:\NumData}  ) \mathrm{d}\lnfv_{1:\NumData}$ denotes the expectation concerning $\lnfv_{1:\NumData}$ that follows $ \proposalPDF( \lnfv_{1:\NumData}  )$.
Each $\lnfv_{1:\NumData}^{(\IDMC)}$ for $\IDMC\in \{1,2,\dots,\NumMC\}$ is an i.i.d. sample from $\proposalPDF( \lnfv_{1:\NumData} )$.
\end{remark}

A challenge is to design a proposal PDF $\proposalPDF( \lnfv_{1:\NumData} )$ that serves as a proper surrogate for $ \PDF( \lnfv_{1:\NumData}  | {\Dataset{}} )$.
Supposing that the number of observations $\Output_{\IDt,\IDsamp}$ in \eqref{eq:def_output} is greater than 1, that is, $\NumSamp>1$, 
this study proposes the following $\proposalPDF( \lnfv_{1:\NumData} )$:
\begin{align}
	\proposalPDF( \lnfv_{1:\NumData}    )
	=
	\mathcal{N}( \lnfv_{1:\NumData} | \gpmlnfv , 	\gpcovlnfv   )
	,\label{eq:def_proposalPDF}
\end{align}
where the following definitions are utilized, with the $\NumData$-dimensional identify matrix denoted by $\Identity$:
\begin{align}
	\gpmlnfv
	&:=
	\kGramfv_{\NumData}  (\kGramfv_{\NumData}  + \lnfvvar \Identity )^{-1}
	{\slnfv{1:\NumData}}
	,\nonumber
	\\
	\gpcovlnfv
	&:=
	\kGramfv_{\NumData}  -  \kGramfv_{\NumData}^{\MyTop} (\kGramfv_{\NumData}  + \lnfvvar \Identity )^{-1}  \kGramfv_{\NumData} 
	,\nonumber
	\\
	\lnfvvarSQRT
	&:=	
	\sqrt{ {\polygamma{1}}\Big(\frac{\NumSamp-1}{2}\Big) }
	,\label{eq:def_lnfvvar}
	\\
	{\slnfv{\IDt}}
	&:=
	\ln {\OutSampVar{\IDt}}  + \ln(\NumSamp-1) -(\ln 2 )- {\polygamma{0}}\Big(\frac{\NumSamp-1}{2}\Big)
	,\label{eq:def_slnfv}
	\\	
	{\OutSampVar{\IDt}}
	&:=
	\frac{1}{\NumSamp-1}
	\sum_{\IDsamp=1}^{\NumSamp}  \Big(
	{\OutSamp{\IDt}{\IDsamp}}  - {\OutSampMean{\IDt}}
	\Big)^{2}
	.\label{eq:OutSampVar}
\end{align}
Note that 
${\polygamma{0}}(\wildcard)$ and ${\polygamma{1}}(\wildcard)$ are the digamma and trigamma functions, respectively, that can be effectively evaluated using MATLAB.
The proposed $\proposalPDF( \lnfv_{1:\NumData}    )$ in \eqref{eq:def_proposalPDF} exhibits two suitable properties discussed in the following.

First, the density ratio $	{ \unPDF( \lnfv_{1:\NumData}^{(\IDMC)}   | {\Dataset{}} ) }/{  \proposalPDF( \lnfv_{1:\NumData}^{(\IDMC)}  )  }$ in \eqref{eq:concept_SNIS} can be computed using multivariate normal distributions.

\begin{proposition}[{\MyHighlight{First property of the proposed $\proposalPDF( \lnfv_{1:\NumData}    )$}}]\label{thm:SNIS_PDF_ratio}
	Suppose Assumption \ref{ass_GPs} and $\NumSamp>1$.
	By setting $\unConstPDF=\PDF(  {\OutSampMean{1:\NumData}}    |  	\Input_{1:\NumData}  )$, we have
	\begin{align}
		\frac{ \unPDF( \lnfv_{1:\NumData}   | {\Dataset{}} ) }{  \proposalPDF( \lnfv_{1:\NumData}  )  }
		&
		=
		\frac{ 
			\mathcal{N}(	{\OutSampMean{1:\NumData}}	| 0,\NkGramfm(\lnfv_{1:\NumData}) )
			\mathcal{N} (\lnfv_{1:\NumData}  | 0 ,\kGramfv_{\NumData} )				
		}{
			\mathcal{N}( \lnfv_{1:\NumData}  | \gpmlnfv , 	\gpcovlnfv   )
		}	
		.\label{eq:PDFratio_result}
	\end{align}
\end{proposition}
\begin{proof}
	The proof is described in Appendix \ref{pf:SNIS_PDF_ratio}.
\end{proof}

Second, the proposed $\proposalPDF( \lnfv_{1:\NumData}    )$ serves as a suitable surrogate for $ \PDF( \lnfv_{1:\NumData}  | {\Dataset{}} )$ in the sense of  the following two approximations:
\begin{align}
	\PDF(  \lnfv_{1:\NumData}  | {\Dataset{}}  )
	=
	\PDF(  \lnfv_{1:\NumData}  | 	\Input_{1:\NumData} ,  {\OutSampMean{1:\NumData}}    )
	&\approx
	\PDF(  \lnfv_{1:\NumData}  | \Input_{1:\NumData} , {\OutSampVar{1:\NumData}} 	 )
	\nonumber\\&\approx
	\proposalPDF( \lnfv_{1:\NumData} )
	\label{eq:proposalPDF_approx_overview}
\end{align}
It is anticipated that the first approximation in \eqref{eq:proposalPDF_approx_overview} is practical because  both ${\OutSampMean{1:\NumData}}$ and ${\OutSampVar{1:\NumData}}$ are derived from the given observations $\Output_{\IDt,\IDsamp}$ using \eqref{eq:OutSampMean} and \eqref{eq:OutSampVar}, respectively.
To substantiate the validity of the second approximation in \eqref{eq:proposalPDF_approx_overview}, we present the following theorem.

\begin{theorem}[{\MyHighlight{Second property of the proposed $\proposalPDF( \lnfv_{1:\NumData}    )$}}]\label{thm:PDF_lnfv1}
	Suppose Assumption \ref{ass_GPs} and $\NumSamp>1$.
	We have	
	\begin{align}
		p( \lnfv_{1:\NumData}   | 	\Input_{1:\NumData}, {\OutSampVar{1:\NumData}}  )
		&		\propto
		p( \lnfv_{1:\NumData} | 	\Input_{1:\NumData}  )
		\prod_{\IDt=1}^{\NumData}
		\LCSdist(  {\slnfv{\IDt}}   | \lnfv_{\IDt}, \lnfvvar )
		,\label{eq:pdf_lnfv_given_y_var}		
		\\
		\proposalPDF( \lnfv_{1:\NumData}   )
		&	\propto
		p( \lnfv_{1:\NumData} | 	\Input_{1:\NumData}  )
		\prod_{\IDt=1}^{\NumData}
		\mathcal{N}({\slnfv{\IDt}} | \lnfv_{\IDt}, \lnfvvar  )
		,	\label{eq:cond_proposalPDF}
	\end{align}
	where
	$\LCSdist(  {\slnfv{\IDt}}   | \lnfv_{\IDt}, \lnfvvar )$
	is the logarithmic chi-squared distribution 
	with the mean $\lnfv_{\IDt}$ and variance $\lnfvvar $ in \eqref{eq:def_lnfvvar}.
\end{theorem}
\begin{proof}
	The proof is described in Appendix \ref{pf:PDF_lnfv1}.
\end{proof}

\begin{remark}[{\MyHighlight{Contributions of Theorem \ref{thm:PDF_lnfv1}}}]
Theorem \ref{thm:PDF_lnfv1} demonstrates that
approximating $p( \lnfv_{1:\NumData}   | 	\Input_{1:\NumData}, {\OutSampVar{1:\NumData}}  )$ by $\proposalPDF( \lnfv_{1:\NumData}   )$ reduces to approximating $\LCSdist(  {\slnfv{\IDt}}   | \lnfv_{\IDt}, \lnfvvar )$ by $\mathcal{N}({\slnfv{\IDt}} | \lnfv_{\IDt}, \lnfvvar  )$, where both $\LCSdist(  {\slnfv{\IDt}}   | \lnfv_{\IDt}, \lnfvvar )$ and $\mathcal{N}({\slnfv{\IDt}} | \lnfv_{\IDt}, \lnfvvar  )$ share the same mean $ \lnfv_{\IDt}$ and variance $\lnfvvar$.
In the sense of the moment matching, $\proposalPDF( \lnfv_{1:\NumData}   )$ is anticipated to serve as a suitable surrogate for $\PDF(  \lnfv_{1:\NumData}  | {\Dataset{}}  )$.
\end{remark}

\section{Control application with demonstration}\label{sec_sim}

\subsection{Chance-constrained tracking control using HGPR}\label{sec_CCtracking}

This study demonstrates tracking control via the proposed HGPR, substantiated by the derived theorems.
Recalling Notation {\notationEl}, let us introduce the two dimensional reference signal ${\sysrV{\MyT}}=[ {\sysr{\MyT}{1}}, {\sysr{\MyT}{2}} ]^{\MyTop} \in \mathbb{R}^{2}$ at $\MyT\in \{0,1,2,\dots\}$ defined by
\begin{align}
	{\sysrV{\MyT+1}}
	&= 
	{\sysrV{\MyT}}
	+
	\MydT
	\begin{bmatrix}
		{\sysr{\MyT}{2}} \\
		{\sysv{\MyT}}
	\end{bmatrix} 
	,
\end{align}
where $\MydT$ is a discrete time period and ${\sysv{\MyT}}$ is a known signal. We control the following partially unknown nonlinear system with the state ${\sysxV{\MyT}}=[ {\sysx{\MyT}{1}}, {\sysx{\MyT}{2}} ]^{\MyTop} \in \mathbb{R}^{2}$ at $\MyT$:
\begin{align}
	{\sysxV{\MyT+1}}
	&= 
	{\sysxV{\MyT}}
	+
	\MydT
	\begin{bmatrix}
		{\sysx{\MyT}{2}} \\
		 \fm({\sysxV{\MyT}})
		+ 	{\sysPreu{\MyT}} 
		  +  {\sysu{\MyT}}
	\end{bmatrix} 
	,\label{eq:def_sys_sim}
	\\
	\sysPreu{\MyT}		&={\sysv{\MyT}}	 - ( {\sysx{\MyT}{1}} - {\sysr{\MyT}{1}} )/\MydT
	,\nonumber
\end{align}
where $\fm({\sysxV{\MyT}})$ denotes an unknown disturbance and $\sysPreu{\MyT}$ is the predefined simple feedforward-feedback controller. 
The unknown function $\fm$ is identified using the proposed HGPR, given a training dataset ${\Dataset{}}$, details of which are provided in the subsequent subsection.
The additional control input $ {\sysu{\MyT}}$ is designed below.

We address the chance-constrained sparse tracking control problem with minimizing the $L1$ norm of the control input sequence:
\begin{align}
&\min_{ {\sysu{\wildcard}} }
\sum_{\MyT=0}^{\Thorizon-1}
\big|{\sysu{\MyT}} \big|
\quad 
\mathrm{s.t.}
\nonumber\\
&
\forall \MyT, \;
\mathrm{Pr}
\Big(
 | {\sysx{\MyT+1}{2}} - {\sysr{\MyT+1}{2}} | \leq \sysrMargin
\Big| {\sysxV{\MyT}} ,  {\Dataset{}}
\Big)\geq 1 - \CDFdeltaSpecial,
\label{eq:CCL1_problem}
\end{align}
where $\sysrMargin$, $\CDFdeltaSpecial$, and $\Thorizon$ are predefined parameters.
While this problem is difficult to solve unfortunately, this study proposes the following controller to achieve control sparsity while simultaneously satisfying the chance constraints.

\begin{theorem}[{\MyHighlight{Chance-constrained sparse control}}]\label{thm:L1_control} 
Given  $\MyT$ and ${\sysxV{\MyT}}$, suppose Assumption \ref{ass_GPs} holds,
that is, $\fm$ and $\lnfv$ obey GPs independent of $(\MyT, {\sysxV{\MyT}})$.
For the given $(\MyT, {\sysxV{\MyT}})$ and a given ${\Dataset{}}$,
the chance constraint for $\MyT$ in \eqref{eq:CCL1_problem} is satisfied if we have 
\begin{align}
	{\sysu{\MyT}}
	&=
	\begin{cases}
	\sparseControllerU & (\sparseControllerU<0)
	\\
	\sparseControllerL & (\sparseControllerL>0)
	\\
	0 & (\mathrm{otherwise})	
	\end{cases}
	,\label{eq:def_L1_controller}
	\\
	\sparseControllerL &\leq \sparseControllerU
	,\label{eq:def_L1_controller_cond}
\end{align}
where
\begin{align}
	\sparseControllerU
	&:=
	-{\CDFlevSpecialU}
	+ (\sysrMargin/\MydT)
	+ \CDFlevSpecialTerm
	,\nonumber\\
	\sparseControllerL
	&:=
	-{\CDFlevSpecialL}
	- (\sysrMargin/\MydT)
	+ \CDFlevSpecialTerm
	,\nonumber
	\\
	\CDFlevSpecialTerm
	&:=
	- {\sysPreu{\MyT}}
	+
	(  {\sysr{\MyT+1}{2}}
	-{\sysx{\MyT}{2}}
	)/\MydT
	,\nonumber
\end{align}	
and
${\CDFlevSpecialU}$ and ${\CDFlevSpecialL}$ are the unique scalars satisfying $\CDFdelta( {\CDFlevSpecialU} ,{\sysxV{\MyT}}, {\Dataset{}}) = \CDFdeltaSpecial/2$ and $\CDFdelta( {\CDFlevSpecialL} ,{\sysxV{\MyT}}, {\Dataset{}}) =1 - \CDFdeltaSpecial/2$.

\end{theorem}
\begin{proof}
The proof is described in Appendix \ref{pf:L1_control}.
\end{proof}
\begin{remark}[{\MyHighlight{Contributions of Theorem \ref{thm:L1_control}}}]
Theorem \ref{thm:L1_control} ensures the chance constraints in \eqref{eq:CCL1_problem} despite the presence of an unknown disturbance $\fm$ under Assumption \ref{ass_GPs}.
The controller \eqref{eq:def_L1_controller} can be sparse because $	{\sysu{\MyT}}=0$ holds as long as $\sparseControllerL\leq 0$ and $\sparseControllerU\geq 0$.
In other words, $\sparseControllerL$ and $\sparseControllerU$ indicate permissible lower and upper bounds of ${\sysu{\MyT}}$ such that the chance constraints are satisfied.	
Note that ${\CDFlevSpecialU}$ and ${\CDFlevSpecialL}$ can be determined by solving $\CDFdelta( {\CDFlevSpecialU} ,{\sysxV{\MyT}}, {\Dataset{}}) = \CDFdeltaSpecial/2$ and $\CDFdelta( {\CDFlevSpecialL} ,{\sysxV{\MyT}}, {\Dataset{}}) =1 - \CDFdeltaSpecial/2$, respectively, using a bisection method.
\end{remark}

\subsection{Simulation setting}\label{sec_sim_setting}

To identify the unknown disturbance $\fm$ in \eqref{eq:def_sys_sim} under heteroscedastic noise associated with $\fv$ as shown in \eqref{eq:def_noise}, we prepare the training dataset as follows.
The inputs $ \Input_{1:\NumData}$ are set to points regularly arrayed on $[-1,1]\times [-1,1]$ with $\NumData=100$.
The number  of observations $\Output_{\IDt,\IDsamp}$ in \eqref{eq:def_output} at each input  $ \Input_{\IDt}$ is set to $\NumSamp=2$.
We consider the following unknown functions:
\begin{align}
	\fm(\Input)&=-10  \sin(\pi {\El{\Input}{1}} )  - 10  \sin(2\pi {\El{\Input}{2}} )  
	,\nonumber\\
	\fv(\Input)^{2}&=0.1+{1.5}/({1+\exp(-10{\El{\Input}{2}} )})
	.\nonumber
\end{align}
In the proposed HGPR, we set $\kerfm(\Input_{\IDEl},\Input_{\IDbEl})$ associated with $\fm$ as the squared exponential kernel:
\begin{align}
\kerfm(\Input_{\IDEl},\Input_{\IDbEl})
=\hypMkerfm \exp\Big(
-\frac{1}{2}(\Input_{\IDEl}-\Input_{\IDbEl})^{\top} 
\begin{bmatrix}
{\hypSkerfm{1}}& 0\\
0 & {\hypSkerfm{2}}	
\end{bmatrix}
(\Input_{\IDEl}-\Input_{\IDbEl})
\Big) 
,\nonumber
\end{align}
where the hyperparameter $(\hypMkerfm,{\hypSkerfm{1}},{\hypSkerfm{2}} )=(
407,
1.37,
5.55
)$ was determined by maximizing a logarithmic marginal likelihood using the \textit{fminunc} command in MATLAB.
The kernel $\kerfv(\Input_{\IDEl},\Input_{\IDbEl})$  associated with $\lnfv=\ln \fv^{2}$ is identical to $\kerfm(\Input_{\IDEl},\Input_{\IDbEl})$ with the hyperparameter set to $(
2.14,
0.0241,
1.86
)$.
The other parameters are set as follows: 
$\MydT=0.005$, 
$\sysrMargin:=0.1$, 
$\CDFdeltaSpecial:=0.01$, 
$\Thorizon=500$,
$\NumMC=1000$,
${\sysxV{0}}={\sysrV{0}}=[0,0]^{\MyTop}$,
and ${\sysv{\MyT}}=3 \cos (\pi \MydT \MyT)$.

We utilize the autonormalized importance sampling \eqref{eq:concept_SNIS} in solving $\CDFdelta( {\CDFlevSpecialU} ,{\sysxV{\MyT}}, {\Dataset{}}) = \CDFdeltaSpecial/2$ and $\CDFdelta( {\CDFlevSpecialL} ,{\sysxV{\MyT}}, {\Dataset{}}) =1 - \CDFdeltaSpecial/2$,  resulting in ${\CDFlevSpecialU}$ and ${\CDFlevSpecialL}$ in Theorem \ref{thm:L1_control}.
Here, 
$\kGramfv_{\NumData}$ and $\gpcovlnfv$ are replaced with 
$\kGramfv_{\NumData}+\JitterVal\Identity$ and 
$\gpcovlnfv+\JitterVal\Identity$ with $\JitterVal=1.0 \times 10^{-12}$, respectively, 
to compute the density ratio \eqref{eq:PDFratio_result}, thereby enhancing numerical stability.

\subsection{Simulation results}\label{sec_sim_results}
\newcommand*{\DirSimResults}{./fig/results20240913}

We compare the  controller proposed in \eqref{eq:def_L1_controller} with the baseline feedback controller
${\sysu{\MyT}}
=-\FBgain (  {\sysx{\MyT}{2}}-{\sysr{\MyT}{2}})/\MydT$, using three types of gains $\FBgain \in \{1,0.5,0.1\}$.
Table \ref{tab:results} shows the values of the control costs  $\sum_{\MyT=0}^{\Thorizon-1} |{\sysu{\MyT}} |$ in \eqref{eq:CCL1_problem} and the violation count by applying the proposed and baseline controllers.
The violation count represents the number of time steps $\MyT \in \{0,1,\dots,\Thorizon - 1\}$ at which the constraint $ | {\sysx{\MyT}{2}} - {\sysr{\MyT}{2}} | \leq \sysrMargin$ in \eqref{eq:CCL1_problem} is not satisfied.
The results indicate that the proposed method outperforms  the baseline controllers.
Figure \ref{fig:control} illustrates the outcomes obtained by applying the proposed method.
Notably, the constraint is nearly satisfied even under the unknown disturbance.
The sequence of the  control inputs is sparse, leading to a reduction in the control cost to the sparse controller proposed in Theorem \ref{thm:L1_control}.

\begin{table}
	\centering
	\renewcommand{\arraystretch}{1.1}	
	\caption{Control costs and violation counts}
	\label{tab:results}
	\begin{tabular}{|c | c|c |c |c |}
		\hline
		
		 \multirow{2}{*}{} & \multirow{2}{*}{\shortstack{ \textbf{Proposed}\\\textbf{controller}}}  & \multicolumn{3}{|c|}{Baseline feedback controllers}   \\		
		\cline{3-5}		
		 &  & { $\FBgain=1$ } &	{ $\FBgain=0.5$ } & {  $\FBgain=0.1$}   \\		\hline
		{Control cost} & \textbf{2953} &{3850}  &{3468} & {1598} 
		\\		\hline
		{Violation count} & \textbf{34} &{0}  &{116} & {349} 
		\\		\hline
	\end{tabular}
	\vspace*{+0.01in}	
\end{table}

\begin{figure}[!t]	
	\vspace*{+0.07in}

		\centering	
		\begin{tikzpicture}[scale=1.0]
			\begin{axis}[axis y line=none, axis x line=none
				,xtick=\empty, ytick=\empty 
				,xmin=-10,xmax=10,ymin=-10,ymax=10
				,width=4 in,height=2.1 in
				]

				\node at (0,0) { 
					\includegraphics[width=0.6\linewidth]{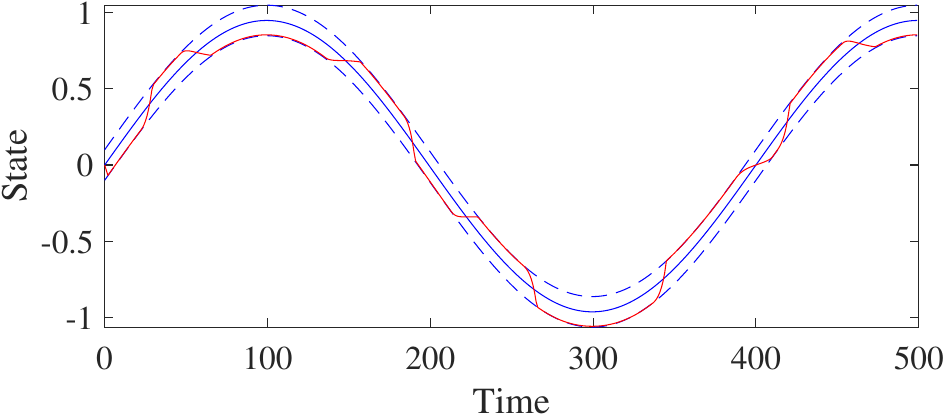}  
				};
				
				\fill[fill=white] (-11,-6) rectangle (-8.8,6);
				\node at (-9.4,2.0) {\rotatebox{90}{	\footnotesize{State ${\sysx{\MyT}{2}}$}	}};
				\fill[fill=white] (-6,-7.8) rectangle (6,-11);
				\node at (0,-9.2) {\footnotesize{Time Step $t$}};		 
			\end{axis}			
		\end{tikzpicture}
		
		\centering	
		\begin{tikzpicture}[scale=1.0]%
			\begin{axis}[axis y line=none, axis x line=none
				,xtick=\empty, ytick=\empty 
				,xmin=-10,xmax=10,ymin=-10,ymax=10
				,width=4 in,height=2.1 in
				]

				\node at (0,0) { 
					\includegraphics[width=0.6\linewidth]{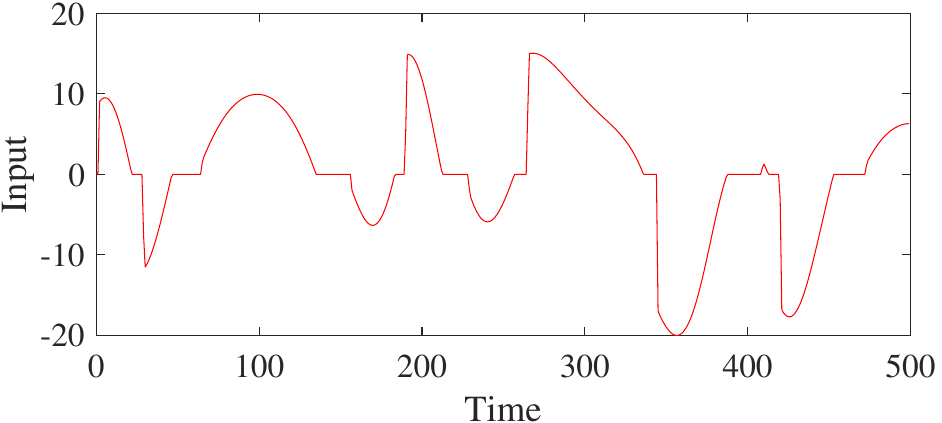}  
				};

				\fill[fill=white] (-11,-6) rectangle (-8.8,6);
				\node at (-9.4,2.0) {\rotatebox{90}{	\footnotesize{Input ${\sysu{\MyT}}$}	}};
				\fill[fill=white] (-6,-7.8) rectangle (6,-11);
				\node at (0,-9.2) {\footnotesize{Time Step $t$}};		 
			\end{axis}			
		\end{tikzpicture}
		
		\caption{Trajectories of the state and input by the proposed method.
		The blue solid and dashed lines denote the reference signal $ {\sysr{\MyT}{2}}$ and $ {\sysr{\MyT}{2}} \pm \sysrMargin$, respectively.
		The red lines in the upper and lower figures indicate the state $ {\sysx{\MyT}{2}}$ and control input ${\sysu{\MyT}}$, respectively.}
		\label{fig:control}
	\end{figure}
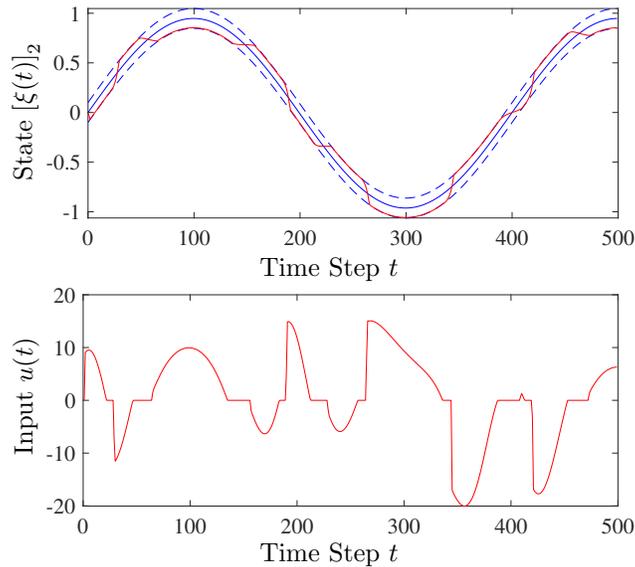

\section{Conclusion}\label{sec_conclusion}

This study presented novel theoretical analyses and the implementation of HGPR.
Under assumptions of HGPR, we derived the posterior mean, variance, and cumulative distribution of a target function.
Despite the computational challenges of the derived expectations, we proposed a sampling-based method for their calculation.
Furthermore, the proposed method was applied to a chance-constrained tracking control problem.
The effectiveness of the proposed controller was confirmed via a numerical simulation.

\appendix

\section{Proof of Theorem \ref{thm:GP_fm}} \label{pf:GP_fm}

Initially, we demonstrate that for every $\Input \in \DomX$,
\begin{align}
	\begin{bmatrix}
		\NkGramfm(\lnfv_{1:\NumData}) 
		&  \kvecfm_{\NumData}(\Input)
		\\	\kvecfm_{\NumData}(\Input)^{\MyTop} 	 & \kerfm(\Input,\Input) 
	\end{bmatrix}
	\succ 0
	,\label{eq:pd_cov_fm}
\end{align}
where this holds clearly for $\Input \in \DomX\setminus \{ \Input_{\IDt} | \IDt \in \{ 1,2,\dots, \NumData \} \}$.
Consider the case of   $\Input \in  \{ \Input_{\IDt} | \IDt \in \{ 1,2,\dots, \NumData \} \}$.
Because Assumption \ref{ass_GPs} implies $	\kGramfm_{\NumData} \succ 0$ and
$	\NkGramfm(\lnfv_{1:\NumData}) 
	-  \kGramfm_{\NumData}^{\MyTop}  \kGramfm_{\NumData}^{-1} \kGramfm_{\NumData} 
	\succ 
	0
$, using Schur complement indicates that for any $\IDt$, for some $\pdTransMat\in\{0,1\}^{ 2 \NumData \times (\NumData+1) }$, we obtain
\begin{align}
	\begin{bmatrix}
		\NkGramfm(\lnfv_{1:\NumData}) 
		&  \kvecfm_{\NumData}(\Input_{\IDt} )
		\\	\kvecfm_{\NumData}(\Input_{\IDt} )^{\MyTop} 	 & \kerfm(\Input_{\IDt} ,\Input_{\IDt} ) 
	\end{bmatrix}
	=
	\pdTransMat^{\MyTop}
	\begin{bmatrix}
		\NkGramfm(\lnfv_{1:\NumData}) 
		&  \kGramfm_{\NumData}
		\\	\kGramfm_{\NumData}	 & \kGramfm_{\NumData}
	\end{bmatrix}
	\pdTransMat
	\succ 0
	.\nonumber
\end{align}
Thus, we have \eqref{eq:pd_cov_fm} for every  $\Input \in \DomX$.

Next, using \eqref{eq:def_noise} and \eqref{eq:def_lnfv} yields
	$
	({1}/{\NumSamp})\sum_{\IDsamp=1}^{\NumSamp} 	\Noise_{\IDt,\IDsamp} 
	| \lnfv_{\IDt} 
	\sim \mathcal{N} (0, 	({1}/{\NumSamp}) \exp \lnfv_{\IDt} )$.
Assumption \ref{ass_GPs} with \eqref{eq:pd_cov_fm} implies
\begin{align}
	\begin{bmatrix}
		{\OutSampMean{1:\NumData}}
		\\
		\fm(\Input)
	\end{bmatrix}
	\Bigg| \Input,\Input_{1:\NumData}, \lnfv_{1:\NumData} 
	&\sim
	\mathcal{N} \Bigg( 0 ,
	\begin{bmatrix}
		\NkGramfm(\lnfv_{1:\NumData}) 
		&  \kvecfm_{\NumData}(\Input)
		\\	\kvecfm_{\NumData}(\Input)^{\MyTop} 	 & \kerfm(\Input,\Input) 
	\end{bmatrix}
	\Bigg)
	,\label{eq:jointPDF}
\end{align}
because 
${\OutSampMean{\IDt}}=\fm(\Input_{\IDt})+ (1/\NumSamp)\sum_{\IDsamp=1}^{\NumSamp} 	\Noise_{\IDt,\IDsamp}$ holds
and
$\fm$ is independent of $\lnfv_{1:\NumData} $ and   $\sum_{\IDsamp=1}^{\NumSamp} 	\Noise_{\IDt,\IDsamp}$.
Utilizing \eqref{eq:jointPDF}, we derive the subsequent relations based on \cite{Rasmussen06}:
\begin{align}
	\opE[\fm(\Input) | \Input, \lnfv_{1:\NumData}, {\Dataset{}}  ]
	&=
	\kvecfm_{\NumData}(\Input)^{\MyTop}  \NkGramfm(\lnfv_{1:\NumData})^{-1}   {\OutSampMean{1:\NumData}}
	,\label{eq:gp_mean_given_h}
	\\
	\opV[\fm(\Input) | \Input, \lnfv_{1:\NumData}, {\Dataset{}}  ]
	&=
	\kerfm(\Input,\Input) -  \kvecfm_{\NumData}(\Input)^{\MyTop} \NkGramfm(\lnfv_{1:\NumData})^{-1}   \kvecfm_{\NumData}(\Input)
	.\label{eq:gp_var_given_h}
\end{align}	
We derive
\begin{align}
	\opE[\fm(\Input) |  \Input, {\Dataset{}}  ]
	&=
	\opE[ \opE[\fm(\Input) |\Input,  \lnfv_{1:\NumData}, {\Dataset{}}  ] | \Input,  {\Dataset{}}   ]
	\nonumber\\&
	=
	\opE[ \kvecfm_{\NumData}(\Input)^{\MyTop}  \NkGramfm(\lnfv_{1:\NumData})^{-1}   {\OutSampMean{1:\NumData}}   |   \Input,{\Dataset{}}   ]
	\nonumber\\&
	=
	\kvecfm_{\NumData}(\Input)^{\MyTop} \opE[ \NkGramfm(\lnfv_{1:\NumData})^{-1}  |   {\Dataset{}}] {\OutSampMean{1:\NumData}}
	,\nonumber
\end{align}	
and
\begin{align}
	&\opV[\fm(\Input) |  \Input, {\Dataset{}}  ]
	\nonumber\\&
	=
	\opE[\opV[ \fm(\Input)| \Input, \lnfv_{1:\NumData}, {\Dataset{}}  ] |  \Input, {\Dataset{}}   ]
	+
	\opV[ \opE[\fm(\Input) | \Input, \lnfv_{1:\NumData}, {\Dataset{}}  ] | \Input,  {\Dataset{}}   ]
	\nonumber\\&
	=
	\opE[ 
	\kerfm(\Input,\Input) -  \kvecfm_{\NumData}(\Input)^{\MyTop} \NkGramfm(\lnfv_{1:\NumData})^{-1}   \kvecfm_{\NumData}(\Input)
	| \Input,  {\Dataset{}}   ]
	\nonumber\\&\quad
	+
	\opV[
	\kvecfm_{\NumData}(\Input)^{\MyTop}  \NkGramfm(\lnfv_{1:\NumData})^{-1}   {\OutSampMean{1:\NumData}}
	| \Input,  {\Dataset{}}   ]
	\nonumber\\&
	=
	\kerfm(\Input,\Input)
	-  \kvecfm_{\NumData}(\Input)^{\MyTop}
	\opE[ 
	\NkGramfm(\lnfv_{1:\NumData})^{-1}  
	|   {\Dataset{}}   ]
	\kvecfm_{\NumData}(\Input)
	\nonumber\\&\quad
	+\opE[
	( \kvecfm_{\NumData}(\Input)^{\MyTop}  \NkGramfm(\lnfv_{1:\NumData})^{-1}   {\OutSampMean{1:\NumData}} )^{2}
	| \Input,  {\Dataset{}}   ]
	\nonumber\\&\quad
	-
	(
	\kvecfm_{\NumData}(\Input)^{\MyTop} \opE[ \NkGramfm(\lnfv_{1:\NumData})^{-1}  |   {\Dataset{}}] {\OutSampMean{1:\NumData}}
	)^{2}
	\nonumber\\&
	=
	\kerfm(\Input,\Input)
	-  \kvecfm_{\NumData}(\Input)^{\MyTop}
	\opE[ 
	\NkGramfm(\lnfv_{1:\NumData})^{-1}  
	|   {\Dataset{}}   ]
	\kvecfm_{\NumData}(\Input)
	\nonumber\\&\quad
	+
	\kvecfm_{\NumData}(\Input)^{\MyTop} 
	\Big(
	\opE[
	\NkGramfm(\lnfv_{1:\NumData})^{-1}   {\OutSampMean{1:\NumData}} {\OutSampMean{1:\NumData}}^{\top} \NkGramfm(\lnfv_{1:\NumData})^{-1} 
	| \Input,  {\Dataset{}}   ]
	\nonumber\\&\quad
	- 
	\opE[ \NkGramfm(\lnfv_{1:\NumData})^{-1}  |   {\Dataset{}}] 
	{\OutSampMean{1:\NumData}} {\OutSampMean{1:\NumData}}^{\top}
	\opE[ \NkGramfm(\lnfv_{1:\NumData})^{-1}  |   {\Dataset{}}] 	
	\Big)
	\kvecfm_{\NumData}(\Input)
	\nonumber\\&
	=
	\kerfm(\Input,\Input)
	+  \kvecfm_{\NumData}(\Input)^{\MyTop}
	\gpsdGramCollection( {\Dataset{}})
	\kvecfm_{\NumData}(\Input)
	.
\end{align}
This completes the proof.	
\hfill$\blacksquare$

\section{Proof of Theorem \ref{thm:GP_prob_fm}} \label{pf:GP_prob_fm}

Using {Fubini's theorem}, we obtain
\begin{align}
	&
	\mathrm{Pr}(      \fm(\Input)   > \CDFlev   |  \Input, {\Dataset{}})
	=
	\int_{\CDFlev}^{\infty} 
	p(\fm(\Input) |  \Input, {\Dataset{}} )
	\mathrm{d} \fm(\Input)	
	\nonumber\\&
	=
	\int_{\CDFlev}^{\infty} 
	\Big(
	\int
	p(\fm(\Input) |  \Input, {\Dataset{}}, \lnfv_{1:\NumData}  )
	p( \lnfv_{1:\NumData}  |  \Input, {\Dataset{}})
	\mathrm{d}  \lnfv_{1:\NumData} 
	\Big)
	\mathrm{d} \fm(\Input)	
	\nonumber\\&
	=
	\int
	\Big(
	\int_{\CDFlev}^{\infty}  
	p(\fm(\Input) |  \Input, {\Dataset{}}, \lnfv_{1:\NumData}  )
	\mathrm{d} \fm(\Input)
	\Big)
	p( \lnfv_{1:\NumData}  |  \Input, {\Dataset{}})
	\mathrm{d}  \lnfv_{1:\NumData} 
	\nonumber\\&
	=
	\int
	\mathrm{Pr}(      \fm(\Input)   > \CDFlev   |  \Input, {\Dataset{}}, \lnfv_{1:\NumData}  )
	p( \lnfv_{1:\NumData}  |  \Input, {\Dataset{}})
	\mathrm{d}  \lnfv_{1:\NumData} 	
	. \label{eq:Pr_using_Fubini}
\end{align}
Let $\gpmGhfm(\Input):= \kvecfm_{\NumData}(\Input)^{\MyTop}  \NkGramfm(\lnfv_{1:\NumData})^{-1}   {\OutSampMean{1:\NumData}}$ and 
$\gpsdGhfm(\Input)^{2}:=  \kerfm(\Input,\Input) -  \kvecfm_{\NumData}(\Input)^{\MyTop}\NkGramfm(\lnfv_{1:\NumData})^{-1}\kvecfm_{\NumData}(\Input)  $.
Utilizing \eqref{eq:gp_mean_given_h} and \eqref{eq:gp_var_given_h}, we derive
$p(\fm(\Input) |  \Input, {\Dataset{}}, \lnfv_{1:\NumData}  )
=
\mathcal{N}(\fm(\Input) |   \gpmGhfm(\Input) , \gpsdGhfm(\Input)^{2}  )$. 
Let $\CDFlevStandardNormal:=(\CDFlev - \gpmGhfm(\Input))/\gpsdGhfm(\Input)$ such that $\CDFlev= \gpmGhfm(\Input) + \CDFlevStandardNormal \gpsdGhfm(\Input)$.
We derive
\begin{align}
	&
	\mathrm{Pr}(      \fm(\Input)   > \CDFlev   |  \Input, {\Dataset{}}, \lnfv_{1:\NumData}  )
	\nonumber\\
	&=
	\int_{ \gpmGhfm(\Input) + \CDFlevStandardNormal \gpsdGhfm(\Input) }^{\infty}
	\mathcal{N}(\fmForInt |   \gpmGhfm(\Input) , \gpsdGhfm(\Input)^{2}  )
	\mathrm{d} \fmForInt
	\nonumber\\&
	=
	\int_{\CDFlevStandardNormal }^{\infty}
	\mathcal{N}(\fmForInt |  0, 1  )
	\mathrm{d} \fmForInt
	=:
	\Qfunc(\CDFlevStandardNormal)
	, \label{eq:Q_func_pf1}
\end{align}
where $\Qfunc$ is the Gaussian Q function satisfying
\begin{align}
	\Qfunc(\CDFlevStandardNormal)
	&
	=
	\int_{\CDFlevStandardNormal}^{\infty}
	\frac{1}{\sqrt{2\pi}} \exp(- \fmForInt^{2}/2)
	\mathrm{d} \fmForInt
	=
	\frac{1}{2}\erfc(\CDFlevStandardNormal/\sqrt{2})
	,
	\label{eq:Q_func_pf2}	
\end{align}
Substituting \eqref{eq:Q_func_pf2} into \eqref{eq:Q_func_pf1} results in
\begin{align}
	\mathrm{Pr}(      \fm(\Input)   > \CDFlev   |  \Input, {\Dataset{}}, \lnfv_{1:\NumData}  )
	&=
	\frac{1}{2}\erfc\Big(
	\frac{ \CDFlev - \gpmGhfm(\Input) }{\sqrt{2} \gpsdGhfm(\Input) }
	\Big)
	.  \label{eq:Q_func_pf3} 
\end{align}
Substituting \eqref{eq:Q_func_pf3} into \eqref{eq:Pr_using_Fubini} results in
\begin{align}
	&
	\mathrm{Pr}(      \fm(\Input)   > \CDFlev   |  \Input, {\Dataset{}})
	\nonumber\\
	&
	=
	\int
	\frac{1}{2}\erfc\Big(
	\frac{ \CDFlev - \gpmGhfm(\Input) }{\sqrt{2} \gpsdGhfm(\Input) }
	\Big)
	p( \lnfv_{1:\NumData}  |  \Input, {\Dataset{}})
	\mathrm{d}  \lnfv_{1:\NumData} 	
	\nonumber\\&
	=\CDFdelta(\CDFlev,\Input, {\Dataset{}}) 
	.\nonumber
\end{align}
This completes the proof.	
\hfill$\blacksquare$

\section{Proof of Proposition \ref{thm:SNIS_PDF_ratio}} \label{pf:SNIS_PDF_ratio}

	Because of $\unConstPDF=\PDF(  {\OutSampMean{1:\NumData}}    |  	\Input_{1:\NumData}  )$,
using Bayes' theorem with \eqref{eq:jointPDF} under Assumption \ref{ass_GPs}  yields
\begin{align}
	\unPDF( \lnfv_{1:\NumData}   |  {\OutSampMean{1:\NumData}} , 	\Input_{1:\NumData}  )
	&=
	\unConstPDF
	\frac{
		\PDF(  {\OutSampMean{1:\NumData}}    |  \lnfv_{1:\NumData}, 	\Input_{1:\NumData}  )
		\PDF( \lnfv_{1:\NumData} | 	\Input_{1:\NumData}  )
	}{
		\PDF(  {\OutSampMean{1:\NumData}}    |  	\Input_{1:\NumData}  )
	}
	\nonumber\\&
	= 
	\PDF(  {\OutSampMean{1:\NumData}}    |  \lnfv_{1:\NumData}, 	\Input_{1:\NumData}  )
	\PDF( \lnfv_{1:\NumData} | 	\Input_{1:\NumData}  )
	\nonumber\\&
	=
	\mathcal{N}(	{\OutSampMean{1:\NumData}}	| 0,\NkGramfm(\lnfv_{1:\NumData}) )
	\mathcal{N} (\lnfv_{1:\NumData}  | 0 ,\kGramfv_{\NumData} )		
	.\nonumber
\end{align}
Combining this equality with \eqref{eq:def_proposalPDF} yields \eqref{eq:PDFratio_result}.
This completes the proof.
\hfill$\blacksquare$

\section{Proof of Theorem \ref{thm:PDF_lnfv1}} \label{pf:PDF_lnfv1}
Initially, we demonstrate   \eqref{eq:pdf_lnfv_given_y_var}.
Utilizing  Bayes' theorem yields
\begin{align}
	p( \lnfv_{1:\NumData}   | 	\Input_{1:\NumData}, {\OutSampVar{1:\NumData}}  )
	\propto
	p( {\OutSampVar{1:\NumData}}   |  \lnfv_{1:\NumData},	\Input_{1:\NumData}  )
	p( \lnfv_{1:\NumData}   | 	\Input_{1:\NumData} )	
	.	\label{eq:Bayes_lnfv_given_dat}
\end{align}
We define
${\chiSamp{\IDt}} := (\NumSamp-1) {\OutSampVar{\IDt}}/( \fv(\Input_{\IDt})^{2})$.
We derive
\begin{align}
	{\chiSamp{\IDt}} 
	&=
	\sum_{\IDsamp=1}^{\NumSamp} 
	\frac{
		\Big(
		(\fm(\Input_{\IDt}) + \Noise_{\IDt,\IDsamp})
		- 
		\Big(
		\frac{1}{\NumSamp}\sum_{\IDbsamp=1}^{\NumSamp} (\fm(\Input_{\IDt}) + \Noise_{\IDt,\IDbsamp})
		\Big)
		\Big)^{2}
	}{
		\fv(\Input_{\IDt})^{2}
	}
	\nonumber\\&
	=
	\sum_{\IDsamp=1}^{\NumSamp}  \Big(		
	\frac{\Noise_{\IDt,\IDsamp}}{\fv(\Input_{\IDt})}		
	- 
	\Big(
	\frac{1}{\NumSamp}\sum_{\IDbsamp=1}^{\NumSamp} 
	\frac{\Noise_{\IDt,\IDbsamp}}{\fv(\Input_{\IDt})}		
	\Big)
	\Big)^{2}
	.\nonumber
\end{align}	
Recall that $\Noise_{\IDt,\IDsamp}/\fv(\Input_{\IDt})$ follows 
$\mathcal{N}(0,1)$ and is independent of $(\IDt,\IDsamp,\Input_{\IDt},\fv(\Input_{\IDt}))$.
Applying Cochran's theorem \cite{Cochran18}, 
for each $\IDt$, ${\chiSamp{\IDt}} $ follows the chi-squared distribution ${\chiPDF{\NumSamp-1}}$ with $({\NumSamp-1})$ degrees of freedom.
Subsequently, the logarithm of $ {\OutSampVar{\IDt}} $ is expressed as
\begin{align}
	\ln {\OutSampVar{\IDt}} 
	= \ln \frac{ \fv(\Input_{\IDt})^{2} }{\NumSamp-1}  {\chiSamp{\IDt}}  
	&= \ln \fv(\Input_{\IDt})^{2} - \ln(\NumSamp-1) + \ln {\chiSamp{\IDt}}
	\nonumber\\&
	= \lnfv_{\IDt} - \ln(\NumSamp-1) + \ln {\chiSamp{\IDt}}
	.\nonumber
\end{align}
Here, utilizing $\opE[\ln {\chiSamp{\IDt}}]=\ln 2+ {\polygamma{0}}((\NumSamp-1)/2)$ \cite[Eq. (1)]{arxiv15} and the definition \eqref{eq:def_slnfv} of ${\slnfv{\IDt}}$ yields
\begin{align}
	{\slnfv{\IDt}}
	&=
	\ln {\OutSampVar{\IDt}}  + \ln(\NumSamp-1) - \opE[\ln {\chiSamp{\IDt}} ]
	=\lnfv_{\IDt}  + \ln {\chiSamp{\IDt}} - \opE[\ln {\chiSamp{\IDt}} ]
	.\nonumber
\end{align}
Because $ {\chiSamp{\IDt}} $ is independent of ${\lnfv_{\IDt}}$,  we derive
\begin{align}
	\opE[{\slnfv{\IDt}} |  {\lnfv_{\IDt}} ]
	&=
	\opE[ {\lnfv_{\IDt}}  |   {\lnfv_{\IDt}} ]
	+
	\opE[\ln {\chiSamp{\IDt}}  |   {\lnfv_{\IDt}}  ]
	-
	\opE[\ln {\chiSamp{\IDt}} ]
	= {\lnfv_{\IDt}}
	.\nonumber
\end{align}
Using the relation $\opE[(\ln {\chiSamp{\IDt}})^{2} ] - \opE[\ln {\chiSamp{\IDt}} ]^{2} ={\polygamma{1}}((\NumSamp-1)/2))$ in \cite[Eq. (1)]{arxiv15} yields
\begin{align}
	\opV[ {\slnfv{\IDt}}   | {\lnfv_{\IDt}} ]
	&=
	\opE[ (  {\lnfv_{\IDt}}+\ln {\chiSamp{\IDt}} - \opE[\ln {\chiSamp{\IDt}} ] )^2   |  {\lnfv_{\IDt}}  ]
	\nonumber\\&\qquad
	- \opE[   {\lnfv_{\IDt}}+\ln {\chiSamp{\IDt}} - \opE[\ln {\chiSamp{\IDt}}]  |   {\lnfv_{\IDt}}  ]^{2}
	\nonumber\\&
	=
	\opE[(\ln {\chiSamp{\IDt}})^{2} ] - \opE[\ln {\chiSamp{\IDt}} ]^{2}
	= 
	\lnfvvar 
	.\nonumber
\end{align}	
Therefore, we have $p(  {\slnfv{\IDt}}   |  \lnfv_{\IDt}  )=	\LCSdist(  {\slnfv{\IDt}}   | \lnfv_{\IDt}, \lnfvvar )$.	
Finally, because the logarithm is monotonically increasing, we obtain the following based on the change of variables at ${\OutSampVarForChangeOfVar= {\OutSampVar{\IDt}} }$: 
\begin{align}
	p( {\OutSampVar{1:\NumData}}   |  \lnfv_{1:\NumData},	\Input_{1:\NumData}  )
	&=
	\prod_{\IDt=1}^{\NumData}
	p( {\OutSampVar{\IDt}}   |  \lnfv_{\IDt} )
	=
	\prod_{\IDt=1}^{\NumData}
	p(  {\slnfv{\IDt}}   |  \lnfv_{\IDt}  )
	\Big|
	\frac{\mathrm{d}}{  \mathrm{d}\OutSampVarForChangeOfVar } \ln\OutSampVarForChangeOfVar
	\Big|
	.	\nonumber
\end{align}
Substituting this relation and $p(  {\slnfv{\IDt}}   |  \lnfv_{\IDt}  )=	\LCSdist(  {\slnfv{\IDt}}   | \lnfv_{\IDt}, \lnfvvar )$ into \eqref{eq:Bayes_lnfv_given_dat} yields \eqref{eq:pdf_lnfv_given_y_var}.

Next, we demonstrate \eqref{eq:cond_proposalPDF}.
Because Assumption \ref{ass_GPs} implies $	\kGramfv_{\NumData} \succ 0$ and
thus 
$(\kGramfv_{\NumData}  + \lnfvvar \Identity )
-  \kGramfv_{\NumData}^{\MyTop}  \kGramfv_{\NumData}^{-1} \kGramfv_{\NumData} 
\succ 
0
$,
using Schur complement implies 
\begin{align}
	\begin{bmatrix}
		\kGramfv_{\NumData}  +\lnfvvar \Identity
		&  \kGramfv_{\NumData}
		\\	\kGramfv_{\NumData}	 & \kGramfv_{\NumData}
	\end{bmatrix}
	\succ 0
	.\nonumber
\end{align}
Suppose that some $({\slnfv{1:\NumData}}, {\lnfv_{1:\NumData}})$ does not satisfy:
\begin{align}
	p\Big(\begin{bmatrix}
		{\slnfv{1:\NumData}}
		\\
		{\lnfv_{1:\NumData}}
	\end{bmatrix}
	\Big| \Input_{1:\NumData}
	\Big)
	&=
	\mathcal{N} \Bigg( 
	\begin{bmatrix}
		{\slnfv{1:\NumData}}
		\\
		{\lnfv_{1:\NumData}}
	\end{bmatrix}
	\Bigg|0 ,
	\begin{bmatrix}
		\kGramfv_{\NumData}  +\lnfvvar \Identity
		&  \kGramfv_{\NumData}
		\\	\kGramfv_{\NumData}	 & \kGramfv_{\NumData}
	\end{bmatrix}
	\Bigg)
	\nonumber\\&\quad\times \PDFratio
	,\label{eq:jointPDF_lnfv_ass1}
\end{align}
where
$\PDFratio:=
p( {\slnfv{1:\NumData}} | \Input_{1:\NumData}, {\lnfv_{1:\NumData}} )
/	\mathcal{N}({\slnfv{1:\NumData}} | \lnfv_{1:\NumData}, \lnfvvar \Identity  )
$.
Then, such $({\slnfv{1:\NumData}}, {\lnfv_{1:\NumData}})$ cannot satisfy
\begin{align}
	&p({\slnfv{1:\NumData}}
	| \Input_{1:\NumData},{\lnfv_{1:\NumData}})
	\nonumber\\&
	=
	\mathcal{N} \big({\slnfv{1:\NumData}}\big|
	\kGramfv_{\NumData}\kGramfv_{\NumData}^{-1}{\lnfv_{1:\NumData}}  ,
	(\kGramfv_{\NumData}  + \lnfvvar \Identity)  
	-  \kGramfv_{\NumData}^{\MyTop}  \kGramfv_{\NumData}^{-1}  \kGramfv_{\NumData} 			
	\big)
	\nonumber\\&\quad\times \PDFratio
	\nonumber\\&
	=
	\mathcal{N} \big({\slnfv{1:\NumData}}\big|
	{\lnfv_{1:\NumData}}  ,		\lnfvvar \Identity		\big)
	\PDFratio
	.\nonumber
\end{align}
As this contradicts the definitions of $\PDFratio$,
we derive \eqref{eq:jointPDF_lnfv_ass1} for every $({\slnfv{1:\NumData}}, {\lnfv_{1:\NumData}})$.
Therefore, we derive the following relations based on \cite{Rasmussen06}:
\begin{align}
	p({\lnfv_{1:\NumData}}
	| \Input_{1:\NumData}, {\slnfv{1:\NumData}})
	&=
	\mathcal{N} \big( {\lnfv_{1:\NumData}}\big|
	\gpmlnfv  ,
	\gpcovlnfv
	\big)
	\PDFratio
	\nonumber\\&
	=
	\proposalPDF( \lnfv_{1:\NumData}    )\PDFratio
	.\nonumber
\end{align}
This result leads to \eqref{eq:cond_proposalPDF} based on Bayes' theorem:
\begin{align}
	\proposalPDF( \lnfv_{1:\NumData}    )
	&=
	p({\lnfv_{1:\NumData}}
	| \Input_{1:\NumData}, {\slnfv{1:\NumData}})/\PDFratio
	\nonumber\\&
	\propto
	p( \lnfv_{1:\NumData} | 	\Input_{1:\NumData}  )
	p(  {\slnfv{1:\NumData}} | \lnfv_{1:\NumData} ,	\Input_{1:\NumData}  )/\PDFratio
	\nonumber\\&
	=
	p( \lnfv_{1:\NumData} | 	\Input_{1:\NumData}  )
	\prod_{\IDt=1}^{\NumData}
	\mathcal{N}({\slnfv{\IDt}} | \lnfv_{\IDt}, \lnfvvar  )
	.\nonumber
\end{align}		
This completes the proof.	
\hfill$\blacksquare$
	
\section{Proof of Theorem \ref{thm:L1_control}} \label{pf:L1_control}
Utilizing \eqref{eq:def_sys_sim}, we derive 
\begin{align}
	&|{\sysx{\MyT+1}{2}} - {\sysr{\MyT+1}{2}}| \leq \sysrMargin
	\nonumber\\&
	\Leftrightarrow
	|
	{\sysx{\MyT}{2}}
	+
	\MydT ( 
	\fm({\sysxV{\MyT}})
	+ 	{\sysPreu{\MyT}} 
	+  {\sysu{\MyT}}
	)
	-	 {\sysr{\MyT+1}{2}} 
	|
	\leq \sysrMargin
	\nonumber\\&
	\Leftrightarrow
	|
	\fm({\sysxV{\MyT}})
	- \CDFlevSpecialTerm  
	+  {\sysu{\MyT}}
	|
	\leq \sysrMargin/\MydT
	\nonumber\\&
	\Leftrightarrow
	- (\sysrMargin/\MydT) + \CDFlevSpecialTerm - {\sysu{\MyT}}
	\leq
	\fm({\sysxV{\MyT}})
	\leq
	(\sysrMargin/\MydT) + \CDFlevSpecialTerm - {\sysu{\MyT}}
	.\nonumber
\end{align}
Utilizing Corollary \ref{thm:GP_prob_fm_special} and its proof, we derive
\begin{align}
	&\mathrm{Pr}(  |{\sysx{\MyT+1}{2}} - {\sysr{\MyT+1}{2}}| \leq \sysrMargin   |  {\sysxV{\MyT}}, {\Dataset{}})
	\nonumber\\&
	=
	\mathrm{Pr}\Big(   
	- (\sysrMargin/\MydT) + \CDFlevSpecialTerm - {\sysu{\MyT}}
	\leq
	\fm({\sysxV{\MyT}})
	\nonumber\\&\qquad\qquad\qquad\qquad\qquad\quad
	\leq
	(\sysrMargin/\MydT) + \CDFlevSpecialTerm - {\sysu{\MyT}}   
	\Big|  {\sysxV{\MyT}}, {\Dataset{}}
	\Big)
	\nonumber\\&
	=
	\CDFdelta(  	- (\sysrMargin/\MydT) + \CDFlevSpecialTerm - {\sysu{\MyT}}  ,{\sysxV{\MyT}}, {\Dataset{}})
	\nonumber\\&\qquad
	-\CDFdelta(       (\sysrMargin/\MydT) + \CDFlevSpecialTerm - {\sysu{\MyT}}  ,{\sysxV{\MyT}}, {\Dataset{}}) 
	\nonumber\\&
	\geq
	\CDFdelta(  \CDFlevSpecialL ,{\sysxV{\MyT}}, {\Dataset{}})
	-\CDFdelta( \CDFlevSpecialU  ,{\sysxV{\MyT}}, {\Dataset{}}) 
	\nonumber\\&
	= 1- \CDFdeltaSpecial
	,\label{eq:anotherRep_CC}
\end{align}	
if we have
$	- (\sysrMargin/\MydT) + \CDFlevSpecialTerm - {\sysu{\MyT}}  \leq \CDFlevSpecialL$
and
$	(\sysrMargin/\MydT) + \CDFlevSpecialTerm - {\sysu{\MyT}}  \geq \CDFlevSpecialU$.
These inequalities are equivalent to 
\begin{align}
	\sparseControllerL
	=
	- \CDFlevSpecialL	- (\sysrMargin/\MydT) + \CDFlevSpecialTerm 
	\leq  {\sysu{\MyT}} \leq 
	-\CDFlevSpecialU +	(\sysrMargin/\MydT) + \CDFlevSpecialTerm 
	=
	\sparseControllerU
	.\label{eq:sufficientCond_CC}
\end{align}
Thus, using \eqref{eq:def_L1_controller} and \eqref{eq:def_L1_controller_cond} satisfies \eqref{eq:sufficientCond_CC} and thus \eqref{eq:anotherRep_CC}.
This completes the proof.
\hfill$\blacksquare$

\end{document}